\documentclass[11pt]{article}
\usepackage{a4}
\usepackage{amsmath,amssymb,amsthm, graphicx}
\usepackage{wasysym}
\usepackage[dvips]{epsfig}
\usepackage{amsfonts}

\newtheorem{theorem}{Theorem}[section]
\newtheorem{corollary}[theorem]{Corollary}
\newtheorem{lemma}[theorem]{Lemma}

\newcommand{\R}{\mathbb{R}}
\newcommand{\N}{\mathbb{N}}
\newcommand{\K}{{\sf K}}
\renewcommand{\L}{\mathcal{L}}
\newcommand{\F}{\mathcal{F}}
\newcommand{\M}{\mathcal{M}}

\newcommand{\heading}[1]{\medskip\par\noindent{\bf #1}}

\DeclareMathOperator{\conv}{conv}

\title{Non-representability of finite projective planes by convex sets}
\author{
Martin Tancer\thanks{
Department of Applied Mathematics and Institute for Theoretical Computer Science (supported by project 1M0545
of The Ministry of Education of the Czech Republic), Faculty of Mathematics
and Physics, Charles University, Malostransk\'e n\'am.~25, 118~00 Prague,
Czech Republic. Partially supported by project GAUK 49209.
E-mail: {\tt tancer@kam.mff.cuni.cz}
}}
\begin{document}

\maketitle
\begin{abstract}
We prove that there is no $d$ such that all finite projective planes
can be represented by convex sets in $\R^d$, answering a question of
Alon, Kalai, Matou\v{s}ek, and Meshulam.
Here, if $\mathbb P$ is a projective
plane with lines $\ell_1,\ldots,\ell_n$, a \emph{representation of $\mathbb P$
by convex sets} in $\R^d$ is a collection of convex sets $C_1,\ldots,C_n
\subseteq \R^d$ such that $C_{i_1},C_{i_2}\ldots,C_{i_k}$ have
a common point if and only if the corresponding lines
$\ell_{i_1},\ldots,\ell_{i_k}$ have a common point in $\mathbb P$. The proof combines a
positive-fraction
selection lemma of Pach with a result of Alon on ``expansion''
of finite projective planes. As a corollary, we show that
for every $d$ there are 2-collapsible simplicial complexes that are
not $d$-representable, strengthening a result of Matou\v{s}ek and the author.

%


\end{abstract}

\section{Introduction}
\heading{Intersection patterns of convex sets.}
One of the important areas in discrete geometry regards understanding to
``intersection patterns'' of convex sets. A~pioneering result in this area is 
Helly's theorem \cite{helly23} which
asserts that if $C_1,C_2,\ldots,C_n$
are convex sets in $\R^d$, $n\ge d+1$ and every $d+1$
of the $C_i$ have a common point, then
$\bigcap_{i=1}^n C_i\ne \emptyset$. 

Consequently, intersection patterns of convex sets have been studied intensively.
This study led to the introduction of $d$-representable and $d$-collapsible
simplicial complexes. We recall that the \emph{nerve} of a family
$\mathcal S = \{S_1, S_2, \dots, S_n\}$ is the
simplicial complex with vertex set $[n]:=\{1,2,\ldots,n\}$
and with a set $\sigma\subseteq [n]$ forming a simplex
if $\bigcap_{i\in \sigma} S_i\ne\emptyset$.
A simplicial complex $\K$ is
\emph{$d$-representable} if it is
isomorphic to the nerve of a family of convex sets in $\R^d$.
In this language Helly's theorem
implies that a $d$-representable complex is determined
by its $d$-skeleton.
We refer to \cite{danzer-grunbaum-klee63,eckhoff93,matousek02,matousek-tancer08} for more
examples and background.

Wegner in his seminal 1975 paper \cite{wegner75}
introduced $d$-collapsible simplicial complexes.
To define this notion, we first introduce an
\emph{elementary $d$-collapse}. Let
 $\K$ be a simplicial complex and let $\sigma,\tau\in\K$ be
faces (simplices) such that

(i) $\dim\sigma\le d-1$,

(ii) $\tau$ is an inclusion-maximal face of $\K$,

(iii) $\sigma\subseteq\tau$, and

(iv) $\tau$ is the \emph{only} face of $\K$
satisfying (ii) and (iii).

\noindent
Then we say that $\sigma$ is a \emph{$d$-collapsible face}
of $\K$ and that the simplicial complex
$\K':=\K\setminus\{\eta\in\K: \sigma \subseteq\eta \subseteq \tau\}$
arises from $\K$ by an elementary $d$-collapse.
A~simplicial complex $\K$ is \emph{$d$-collapsible}
 if there exists a sequence of elementary $d$-collapses
that reduces $\K$ to the empty complex $\emptyset$.
Fig.~\ref{FigCol} shows an example of $2$-collapsing.

\begin{figure}
\centering
\mbox{
\includegraphics[width=\textwidth]{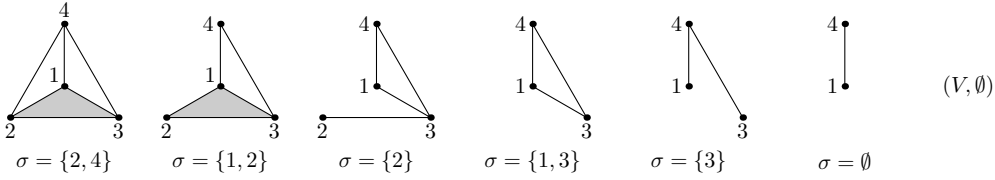}
}
\caption{An example of $2$-collapsing.}
\label{FigCol}
\end{figure}

Wegner \cite{wegner75} showed that every $d$-representable simplicial complex
is $d$-collapsible and also gave an example of 2-collapsible simplicial complex
that is not $2$-representable.

Alon et al.~\cite[Problem~17]{alon-kalai-matousek-meshulam02} asked whether
there is a function $f = f(d)$ such that every $d$-collapsible simplicial
complex is $f(d)$-representable.
Matou\v{s}ek and Tancer~\cite{matousek-tancer08} showed that $f(d) \geq 2d-1$
if
exists.
We improve this result by showing that 
no such $f$ exists.

\heading{Finite projective planes.}
A \emph{finite projective plane of order $q \geq 2$} is a pair $(P, \L)$ where $P$ is
a finite set of \emph{points}, and $\L \subseteq 2^{P}$ is a set of
subsets of $P$ (called \emph{lines}) such that (i) every two points are contained in a unique line,
(ii) every two lines intersect in a unique point, and (iii) every line contains $q
+ 1$ points. It follows that every point is contained in $q + 1$ lines and $|P| =
|\L| = q^2 + q + 1$, see e.g.~\cite{matousek-nesetril98} for more details.

It is well known that a projective plane of order $q$ exists whenever $q$ is a power of a prime. We remark that it is a well known open problem to decide whether there are projective planes of other orders.

It is also known that a finite projective plane cannot be represented in $\R^d$
so that the points of the projective plane are points in $\R^d$ and the lines
of the projective plane are the inclusionwise maximal collections of points
lying on a common Euclidean line. This fact follows for example from
Sylvester-Gallai theorem: if $p_1, \dots, p_n$ are points in the plane not all
of them lying on a common line, then there is a line in the plane which
intersects exactly two of these points (see~\cite{gallai44} for original
solution and \cite{kelly86} for an elegant proof).
Our task is to obtain a similar result where
the Euclidean lines are replaced by convex sets.

Let $(X, \F)$ be a set system where $X$ is a finite set and $\F$ is a set of
some subsets of $X$. We say that $(X, \F)$ is \emph{representable by convex
sets} in $\R^d$ if there are convex sets $C_F \subset \R^d$ for $F \in \F$ such
that for any $F_1, \dots, F_k \in \F$ the convex sets $C_{F_1}, \dots, C_{F_k}$
intersect if and only if the sets $F_1, \dots, F_k$ have a common point in
$X$.\footnote{We strongly distinguish the terms $d$-representable simplicial
complex and a set system representable by convex sets in $\R^d$; they have a
different meaning. In fact, they are dual in a certain sense. A set system $(X,
\F)$ is representable by convex sets in $\R^d$ if and only if the nerve of $\F$ is $d$-representable.}

\begin{theorem}
\label{ThmMain}
For every $d \in \N$ there is a $q_0 = q_0(d) \in \N$ such that any projective
plane $(P, \L)$ of order $q \geq
q_0$ is not representable by convex sets in $\R^d$.
\end{theorem}

We remark that the assumption $q \geq q_0$ cannot be left out since every
projective plane $(P, \L)$ of order $q$ can be easily represented by convex
sets in $\R^{q^2 + q}$. For consider $P$ as the set of vertices of a $(q^2 +
q)$-simplex in $\R^{q^2 + q}$ and set $C_{\ell} := \conv\{p: p \in \ell\}$ for a
line $\ell \in \L$.\footnote{With a bit more effort can be shown that a
projective plane of order $q$ can be represented by convex sets in $\R^{2q+1}$.
This is based on the fact that a simplicial complex of dimension $d$ is $(2d +
1)$-representable.}

Our main tools are the positive-fraction selection lemma and the fact that
the projective planes (considered as bipartite graphs) are
expanders. 

Theorem~\ref{ThmMain} is proved in Section~\ref{SecPrf}.

We also have the following consequence of Theorem~\ref{ThmMain} which answers
the question of Alon~at al.~\cite{alon-kalai-matousek-meshulam02} (as announced
above) and which is proved in Section~\ref{SecGap}. 


\begin{corollary}
\label{ThmGap}
Let $d > 1$ be an integer and let $q_0 = q_0(d)$ be the integer from
Theorem~\ref{ThmMain}. Let $(P, \L)$ be the projective plane of order
$q \geq q_0$. Let $\K_q$ be a simplicial complex whose
vertices are points in $P$ and whose faces are subsets of lines in $\L$. Then
$\K_q$ is 2-collapsible and is not $d$-representable.

\end{corollary}


\section{Proof of the main result}
\label{SecPrf}

\begin{figure}
\begin{center}
\includegraphics{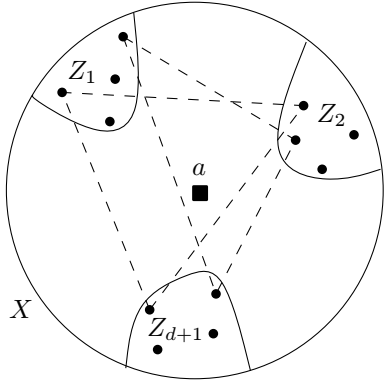}
\end{center}
\caption{Positive-fraction selection lemma: every triangle formed by the sets
$Z_i$ contains $a$.}
\label{FigPFSL}
\end{figure}

In this section we prove Theorem~\ref{ThmMain}. We need few preliminaries. Let
$(Z_1, \dots, Z_k)$ be a $k$-tuple of sets. By a \emph{transversal} of this $k$-tuple
we mean any set $T = \{t_1, \dots, t_k\}$ such that $t_i \in Z_i$ for every $i
\in [k]$. We need the following result due to Pach~\cite{pach98}; see
also~\cite[Theorem 9.5.1]{matousek02}. See Figure~\ref{FigPFSL}.

\begin{theorem}[Positive-fraction selection lemma; a special case]
\label{ThmPFSL}
For every natural number $d$, there exists $c = c(d) > 0$ with the following
property. Let $X \subset \R^d$ be a finite set of points in general position
(i.e., there are no $d + 1$ points lying in a common hyperplane). Then there
is a point $a \in \R^d$ and disjoint subsets $Z_1, \dots, Z_{d+1}$, with $|Z_i|
\geq c|X|$ such that the convex hull of every transversal of $(Z_1, \dots,
Z_{d+1})$ contains $a$.
\end{theorem}

We remark that the proof of Theorem~\ref{ThmPFSL} uses several involved tools
such as weak hypergraph regularity lemma or same-type lemma
(therefore we do not reproduce any details of the proof here). We should also
remark that this is only a special case of Pach's theorem (but general enough);
Pach moreover assumes that $Z_i \subseteq X_i$, where $X_1 \cup \cdots \cup
X_{d+1}$ is a partition of $X$, and in this seting $|Z_i| \geq c|X_i|$.

We also need the following expansion property of the projective
plane~\cite[Theorem 2.1]{alon85}, \cite{alon86}.

\begin{theorem}
\label{ThmExp}
Let $(P, \L)$ be a projective plane of order $q$. Let $A \subseteq P$.
Then $|\{\ell \in \L: \ell \cap A = \emptyset\}| \leq n^{3/2}/|A|$, where $n =
q^2 + q + 1$.
\end{theorem}

Alon, Haussler and Welzl~\cite{alon-haussler-welzl87} used this expansion
property in the context of range-search problems. They showed that the points of a
projective plane (of high enough order) cannot be partitioned into a small
number of sets $P_1$, \dots, $P_m$ so that for every projective line $\ell$ 
the set $\bigcup_{\ell \cap P_i \neq \emptyset} P_i$ contains only a
given fraction of all the points. Known results on range search problems imply
that a projective plane of a high order cannot be represented by half-planes or
simplices in $\R^d$. However, the author is not aware that this approach would
imply the result for convex sets. 

For completeness, we also reproduce a short proof of Theorem~\ref{ThmExp}.

\begin{proof}[Proof of Theorem~\ref{ThmExp}.]
Let $M = (m_{p\ell})$ be an $n \times n$ matrix
with rows indexed by the points of $P$ and columns indexed by the lines of
$\L$. We set $m_{p\ell} := 1$ if $p \in \ell$ and $m_{p\ell} := 0$ otherwise. The
matrix $MM^{T}$ has real nonnegative eigenvalues $\lambda_1 \geq \lambda_2 \geq
\cdots \geq \lambda_n$.

By a theorem of Tanner~\cite{tanner84}
$$
|N(A)| \geq \frac{(q+1)^2 |A|}{((q+1)^2 - \lambda_2) |A| / n + \lambda_2}
$$
where $N(A)$ denotes $\{ \ell \in \L: \ell \cap A \neq \emptyset\}$, the
neighborhood of $A$.

It is not hard to compute that $\lambda_1 = (q + 1)^2$ and $\lambda_2 = \cdots
= \lambda_n = q$. Consequently,
$$
|N(A)| \geq \frac{(q+1)^2 |A|}{|A| + q} = n - \frac{q(n - |A|)}{|A| + q} \geq n -
\frac{n^{3/2}}{|A|}.
$$

\end{proof}

\begin{proof}[Proof of Theorem~\ref{ThmMain}]
For contradiction, we assume that $(P, \L)$ is representable by convex sets in
$\R^d$; i.e., there are convex sets $C_{\ell}$ for $\ell \in \L$ such that
$C_{\ell_1}, \dots, C_{\ell_k}$ intersect if and only if $\ell_1, \dots,
\ell_k$ contain a common point. By
standard tricks, we can assume that these sets are open. We explain this step 
at the end of the proof.

Let $p \in P$. We know that $\bigcap\limits_{p \in
\ell} C_{\ell}$ is nonempty (and open). Let $x_p$ be a point of this
intersection. We define $X := \{x_p : p \in P\}$. Because of the openness of the
intersections we can assume that $X$ is in general position.
\begin{figure}
\begin{center}
\includegraphics{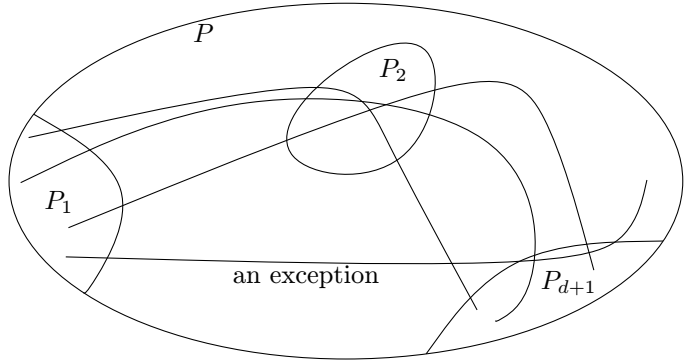}
\end{center}
\caption{Almost all lines intersect all of $P_1, \dots, P_{d+1}$.}
\label{FigLines}
\end{figure}

Let $c = c(d) > 0$, $a \in \R^d$, and $Z_1, \dots, Z_{d+1}$ be the (output)
data from Theorem~\ref{ThmPFSL} (when applied to $X$). We
know that $|Z_i| \geq c|X|$. Let us set $P_i := \{p \in P: x_p \in Z_i\}$ and
$\M_i := \{ \ell \in \L: \ell \cap P_i = \emptyset\}$. 
By Theorem~\ref{ThmExp} with $A = P_i$ we get
$$
|\M_i| \leq \frac{n^{3/2}}{|P_i|}
= \frac{n^{3/2}}{|Z_i|} \leq \frac{n^{3/2}}{c|X|} =
\frac{n^{3/2}}{cn} \leq 
\frac{n}{2(d+1)} = 
\frac{|\mathcal L|}{2(d+1)}
$$
provided that $q$ (and hence $n$ as well) is sufficiently large (depending on $d$ and $c$).

Hence the set $\L' := \L \setminus(\M_1 \cup \dots \cup \M_{d+1})$ of lines
that intersect each of $P_1, \dots, P_{d+1}$ contains at least half of the
lines of $\L$.
Now let $\ell \in \L'$ and let $p_i \in \ell \cap P_i$. Then $(x_{p_1}, \dots,
x_{p_{d+1}})$ is a transversal of $(Z_1, \dots, Z_{d+1})$. Thus $a \in \conv
\{x_{p_1}, \dots, x_{p_{d+1}}\} \subseteq C_\ell$, and so $a$ is contained in at
least $\frac{|\L|}2$ of the $C_\ell$.
This is a contradiction since at most $q + 1$ sets among the $C_{\ell}$ can
have a nonempty intersection.

It remains to show that we can assume the $C_{\ell}$ open.
Let us have a set of lines $\mathcal S = \{\ell_1, \dots, \ell_k\}$ such that
$C(\mathcal S) := \bigcap_{i=1}^k C_{\ell_i}$ is nonempty. 
In such a case we pick a point $y_{\mathcal S} \in C(\mathcal S)$.
For $\ell \in \L$ we set 
$$
C'_{\ell} := \conv\left\{y_{\mathcal S}: \mathcal S \subseteq \L, \ell \in
\mathcal S, C(\mathcal S)
\neq \emptyset \right\}.
$$
In the definition of $C'_{\ell}$ we have considered all the possible
intersection, thus the sets $\{C'_{\ell}: \ell \in \L \}$ have the same
``intersection pattern'' as the sets $\{C_{\ell}: \ell \in \L\}$. It means that
the sets $\{C'_{\ell}: \ell \in \L \}$ are a counterexample for the theorem
if and only if the sets  $\{C_{\ell}: \ell \in \L\}$ are a counterexample. The
sets $C'_{\ell}$ are compact sets it means that since now we can assume that all
the considered sets are compact.

Finally we blow up this sets by a small $\varepsilon > 0$. Thus we get open sets
instead of compact sets.

\end{proof}

\section{Proof of the gap between $d$-representability and $d$-collapsibility}
\label{SecGap}

\begin{proof}[Proof of Theorem~\ref{ThmGap}]
The fact that the complex $\K_q$ is 2-collapsible is essentially mentioned
in~\cite[discussion below Problem~15]{alon-kalai-matousek-meshulam02} (without a proof).

All the inclusionwise maximal faces of $\K_q$ are of the form $\sigma_{\ell} = \{
p: p \in \ell\}$ for $\ell \in \L$. Two such faces intersect only in a vertex,
thus it is possible to 2-collapse these faces gradually to the vertices;
the details are given in Lemma~\ref{LemColl} below. After these collapsings, it is sufficient to remove the
vertices (which are already inclusionwise maximal).

It remains to show that $\K_q$ is not $d$-representable. We consider the dual
projective plane $(\L, \bar P)$, where $\bar P := \{ \{\ell \in \L: p \in \ell
\}: p \in P\}$. In particular we can identify a point $p \in P$ with a
dual line $\{ \ell \in \L: p \in \ell\} \in \bar P$. 
Theorem~\ref{ThmMain} applied for this dual plane $(\L, \bar P)$
essentially states that $\K_q$ is not $d$-representable (convex sets in
the statement now correspond to the lines in $\bar P$, which we have identified
with $P$---the set of vertices of $K_q$).

\end{proof}

\begin{lemma}
\label{LemColl}
Let $\Delta$ be a $d$-simplex, i.e., a simplicial complex with $[d+1]$ as the
set of vertices and with all the possible faces. Then there is a sequence of
elementary 2-collapses that starts with $\Delta$ and ends with the simplicial
complex that contains all the vertices of $\Delta$ and no faces of higher
dimension.
\end{lemma}
\begin{proof}
In every elementary 2-collapse we only mention the smaller face $\sigma$ (here we adopt
the notation from the definition of an elementary $d$-collapse), since a
2-collapse is uniquely determined by $\sigma$.

The following sequence of choices of $\sigma$ provides the required
$2$-collapsing (the faces are ordered in the lexicographical order, see
Figure~\ref{FigCollapsing}).
$$
\{1,2\}, \{1,3\}, \dots, \{1, d+1\}, \{2,3\}, \dots, \{2, d+1\}, \{3, 1\},
\dots, \{d, d +1\}.
$$
\end{proof}

\begin{figure}
\begin{center}
\includegraphics{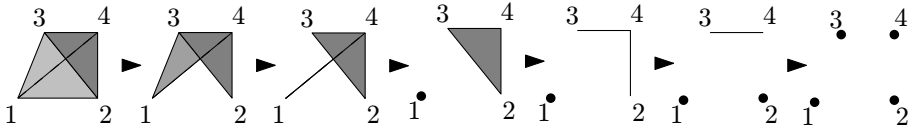}
\end{center}
\caption{Collapsing a simplex to its vertices.}
\label{FigCollapsing}
\end{figure}

\section*{Acknowledgement}
I would like to thank Ji\v{r}\'{\i} Matou\v{s}ek for informing me about
the result on the expansion property of projective planes and also for suggesting many
improving comments while reading the preliminary version of the article.

\bibliographystyle{alpha}
\bibliography{/home/martin/clanky/bib/grph,/home/martin/clanky/bib/topocom,/home/martin/clanky/bib/combgeo,/home/martin/clanky/bib/compl}

\end{document}